\documentclass[12pt]{amsart}
\usepackage{graphicx, enumerate, amsmath, amssymb,amsthm,manfnt,hyperref,amscd,braket}
\usepackage[utf8]{inputenc} 
\allowdisplaybreaks

\newcommand{\R}{{\mathbb{R}}}
\newcommand{\C}{{\mathbb{C}}}
\usepackage{stix}
\usepackage{comment}

\newcommand{\Z}{\mathbb{Z}}

\newcommand{\dbar}{\overline{\partial}}

\newcommand{\zb}{\overline{z}}
\newcommand{\wb}{\overline{w}}

\theoremstyle{plain}
\newtheorem{theorem}[equation]{Theorem}
\newtheorem{proposition}[equation]{Proposition}
\newtheorem{lemma}[equation]{Lemma}
\newtheorem{corollary}[equation]{Corollary}
\newtheorem{definition}[equation]{Definition}
\theoremstyle{remark}
\newtheorem{remark}[equation]{Remark}
\numberwithin{equation}{section}

\title[Stein's theorem on infinite type domains]{Stein's theorem on infinite type domains}
\author{Liwei Chen}
\address[Liwei Chen]{Syracuse University, Department of Mathematics, Syracuse, NY 13244}
\email{lchen125@syr.edu}

\author{Yuan Yuan}
\address[Yuan Yuan]{Syracuse University, Department of Mathematics, Syracuse, NY 13244}
\email{yyuan05@syr.edu}

\subjclass[2010]{Primary: 32A10, 32A50, Secondary: 32T20}
\thanks{The second author is supported by National Science Foundation grant DMS-1412384, Simons Foundation grant (\#429722 Yuan Yuan) and CUSE grant program at Syracuse University}
\keywords{infinite type, Lipschitz continuous, worm domain}
\begin{document}

\maketitle
\begin{abstract}
The disc property is formulated for domains in $\C^n$. Holomorphic Lipschitz functions enjoy a gain in the order of Lipschitz regularity along the complex tangential direction on domains with disc property. 
Disc property is studied on various domains of infinite type.  
As applications, the local version of Stein's theorem is obtained on these domains, including the worm domains. 
\end{abstract}

\section{Introduction}
Stein observed in \cite{S73} that holomorphic Lipschitz $\alpha$ functions on bounded $C^2$ smooth domains in $\C^n$ ($n>1$) are indeed Lipschitz $2\alpha$ in the complex tangential direction near the boundary. From then, this phenomenon has been studied in different aspects and on different types of domains. Rudin proved a similar gain of holomorphic Lipschitz $\alpha$ to Lipschitz $2\alpha$ in the unit ball in \cite{R78}. Krantz extensively studied the Stein's phenomenon in the rather  general situation \cite{Kr80,Kr83,Kr90}. In particular, Krantz uses Kobayashi metric to formulate a general version of Stein's thoerem which contains all the aforementioned phenomena as special cases in \cite{Kr90}. It is also well known that Stein's result is optimal only in the strongly pseudoconvex case (cf. \cite{Kr82}). In \cite{CK91}, Chang and Krantz studied the Stein's phenomenon on finite type domain in $\C^2$, showing that the gain is determined by the type of the boundary point, i.e., a gain from Lipschitz $\alpha$ to Lipschitz $k\alpha$ at a boundary point of type $k$. They also obtained some applications to the $\dbar$-problem on such domains. McNeal and Stein in \cite{MS94} formulated this gain on convex domains of finite type in terms of the size of certain natural polydiscs, which were constructed by McNeal in \cite{Mc94}. In \cite{Ra16}, Ravisankar improved McNeal and Stein's result, showing that this gain can be obtained on domains of finite type that are not necessarily convex.

In the present paper we follow the main theme and continue to investigate Stein's phenomenon on pseudoconvex domains of infinite type.
To formulate our main result, we introduce the disc property of index $k$ (cf. Definition \ref{discproperty}). Our main result (cf. Theorem \ref{main}) in this paper can be stated as a local version of Stein's theorem.

\begin{theorem}[{\bf Main result}]
For $n\ge2$, let $\Omega\subset\C^n$ be a bounded domain with $C^1$ boundary. If $\Omega$ satisfies the disc property of index $k$ along a complex tangential direction $L_p$ at $p\in\partial\Omega$ for some $k\in\Z^+$, then for $0<\alpha<1/k$ any holomorphic Lipschitz $\alpha$ function is indeed Lipschitz $k\alpha$ along the direction $L_p$ at $Q:=p-\delta N_p$, where $N_p$ is the unit outward normal vector at $p$ and $\delta$ is sufficiently small. More precisely,
there exists $c, C>0$, such that 
\[
|f\left(p-\delta N_p\right)-f\left(p-\delta N_p+hL_p\right)|\le C|h|^{k\alpha}
\]
for all $|h| < \frac{c\delta^{1/k}}{10}$ and all $\delta$ sufficiently small.
\end{theorem}   

Certain natural embedded polydiscs idea has already existed in the literature for a long time and has been used in proving many fundamental results (cf. \cite{NSW81, NSW85, FK88, C89, Mc94, MS94}). In the present paper, we formulate this geometric notion with emphasis on the size of the disc. Our main result shows that the Stein's phenomenon indeed follows from this local geometric property--disc property of index $k$. The idea is to obtain estimates of the holomorphic Lipschitz function along the normal and complex tangential directions and then the desired difference quotient estimate can be obtained by splitting the difference into normal and complex tangential directions. 
When the disc property satisfies certain uniform size,  
we  are also able to obtain a uniform version of Stein's theorem (cf. Theorem \ref{semilocalmain}). 

The benefit of discussing the disc property at boundary points is that one can take care of boundary points of infinite type or even infinite regular type on certain domains. With the the main result, we are able to study various domains of infinite type in $\C^2$. In particular, the smooth worm domains are considered in this article. The worm domains are first introduced by Diederich and Forn{\ae}ss in \cite{DF77}, where the worm domains are shown to be bounded smooth pseudoconvex domains without Stein neighborhood basis. In the present paper, it is shown that the worm domains satisfy the disc property of index $k$ at the infinity type boundary points for all $k\in\Z^+$. Indeed, a more general class of domains are considered here (see \S \ref{wormdomains})
and the disc property can be verified. 
As a result, a local version of Stein's theorem at points of infinite type is obtained (see Corollary \ref{steinstheoremonworm} and Remark \ref{steinforallpointsinworm}).

For a domain in $\mathbb{C}^3$, not all boundary points of infinite type enjoy the disc property of index $k$ along all complex tangential directions for each $k\in\Z^+$ (see, for an example, \S \ref{afewremarks} (\ref{C3example1})). In \cite{Ra16}[Example 2.3], Ravisankar considered the domain $\Omega\subset\C^3$ constructed by D'Angelo in \cite{DA82}, where $\Omega$ is locally given by
\[
\Omega\cap U = \left\{ z\in\C^3: r\left(z\right)=\text{Re}\left(z_1\right)+\left|z_2^2-z_3^3\right|^2 <0 \right\}
\]
around the origin $0\in\partial\Omega$. It can be shown that the origin is of infinite type, but of finite regular type. Moreover, the boundary points that are arbitrary close to the origin that lie on the holomorphic curve $\gamma:\zeta\mapsto\left(0,\zeta^3,\zeta^2\right)$, where $0\neq\zeta\in\C\cap V$, are of infinite regular type (cf. \cite{DA82,Ra16}). This makes the result in \cite{Ra16} inapplicable at the boundary points near the origin. Nevertheless, in the present paper, it is shown that $\Omega$ satisfies the disc property of index $k$ at the boundary points that lie outside the line $Z\cap U=\{\text{Re}\left(z_1\right)=z_2=z_3=0\}$, along some complex tangential direction for all $k\in\Z^+$, whereas, the points on the line $Z\cap U$ behave the same as the origin $0$. Hence, our main result not only applies to the origin, but also applies to all the boundary points around the origin $0$ for this domain. This generalizes the result in \cite{Ra16} and provides a solution to the question raised in \cite{Ra16}[Example 2.3] (see \S \ref{afewremarks} (\ref{C3example2})).

The article is organized as follows. In \S \ref{classicalstein}, a review of classical Stein's theorem is given for completeness. In \S \ref{maintheorem}, the disc property is introduced and the main theorem is stated and proven. In \S \ref{wormdomains}, the disc property of worm domains is verified and thus Stein's theorem for worm domains can be obtained. In \S \ref{afewremarks}, the disc property on other domains of infinite type are discussed.

\section{Stein's classical theorem}
\label{classicalstein}
For $n \geq 2$, let $\Omega\subset \mathbb{C}^n$ be a bounded domain with $C^2$ boundary and let $p \in \partial\Omega$. Denote the real unit outward normal vector and unit holomorphic tangent vector at $p$ by $N_p$ and $L_p$ respectively. Denote the bidisc in the two dimensional complex linear subspace $\left(N_p \oplus J N_p\right) \oplus \left(\text{Re}L_p \oplus \text{Im}L_p\right)$  centered at $q\in \Omega$ of bi-radius $\left(a, b\right)$ by $D\left(q; a, b\right)$.
The following is a well known result to experts. We include a detailed proof here for completeness (cf. \cite{Kr82}).

\begin{proposition}
\label{discprop}
Let $n \geq 2$ and $\Omega\subset \mathbb{C}^n$ be a bounded domain with $C^2$ boundary. There exist a sufficiently small $\delta_0 >0$ and a constant $c>0$ such that for any $p \in \partial\Omega$, any holomorphic tangent vector $L_p \in T^{\left(1, 0\right)}_p\left(\partial \Omega\right)$, and any $0<\delta\le\delta_0$, the bidisc $D\left(p-\delta N_p ; c \delta, \sqrt{c\delta}\right)$ centered at $p- \delta N_p $ of bi-radius $\left(c\delta,  \sqrt{c\delta}\right)$ is contained in $\Omega$.
\end{proposition}

\begin{proof}
Given any $p \in \partial\Omega$, let $U$ be an open neighborhood of $p$ such that $\Omega \cap U = \{q\in U: r\left(q\right)<0\}$ for some $C^2$ defining function $r$ with $\text{d}r \left(p\right)\not=0$. 
Choose holomorphic coordinate $z=\left(z_1, \cdots, z_n\right)$ with $z\left(p\right)=0$ and $\frac{\partial r}{\partial z_n}\left(0\right)\not=0$. Then $r$ has the following Taylor expansion at $p$:
\begin{equation}
\begin{split}
r\left(z\right) 
%
&= 2 \text{Re}\left( \sum_{j=1}^n\frac{\partial r}{\partial z_j}\left(0\right) z_j \right)  + \text{Re}\left( \sum_{j, k=1}^n  \left(2- \delta_{jk}\right)  \frac{\partial^2 r}{\partial z_j \partial z_k}\left(0\right) z_j z_k \right)  \\
&\qquad\qquad\qquad\qquad\qquad\qquad\qquad\qquad\qquad + 2 \sum_{j, k=1}^n  \frac{\partial^2 r}{\partial z_j \partial \zb_k}\left(0\right) z_j \zb_k  + \mathcal{O}\left(|z|^3\right), \\
\end{split}   
\end{equation}
where $\varphi\left(z\right)=\mathcal{O}\left(|z|^\alpha\right)$ means $|\varphi\left(z\right)| \leq C |z|^\alpha$ for some $C>0$.
Applying the following change of coordinates: $$w_1=z_1, \cdots, w_{n-1}=z_{n-1}, \text{and}, w_n= 2 \sum_{j=1}^n\frac{\partial r}{\partial z_j}\left(0\right) z_j + \sum_{j, k=1}^n  \left(2 -\delta_{jk}\right) \frac{\partial^2 r}{\partial z_j \partial z_k}\left(0\right) z_j z_k , $$
we have 
\begin{equation}\label{taylor}
\begin{split}
r\left(w\right) &= \text{Re}w_n + 2\sum_{j, k=1}^n  \frac{\partial^2 r}{\partial z_j \partial \zb_{k}} \left(0\right) z_j \zb_k  + \mathcal{O}\left(|z|^3\right) \\
&= \text{Re}w_n + 2\sum_{j, k=1}^{n-1} \frac{\partial^2 r}{\partial w_j \partial \wb_{k}} \left(0\right) w_j \wb_k  + 4 \text{Re} \left( \sum_{j=1}^{n-1} \frac{\partial^2 r}{\partial z_j \partial \zb_n}\left(0\right) z_j \zb_n  \right) \\
&\qquad\qquad\qquad\qquad\qquad\qquad\qquad\qquad\qquad + 2\frac{\partial^2 r}{\partial z_n \partial \zb_n}\left(0\right) |z_n|^2+ \mathcal{O}\left(|z|^3\right)  \\
&\leq  \text{Re}w_n +2 \sum_{j, k=1}^{n-1} \frac{\partial^2 r}{\partial w_j \partial \wb_{k}} \left(0\right) w_j \wb_k + \sum_{j=1}^{n-1} c_j \left( |z_j|^2 + |z_n|^2 \right)\\
&\qquad\qquad\qquad\qquad\qquad\qquad\qquad\qquad\qquad + 2\frac{\partial^2 r}{\partial z_n \partial \zb_n}\left(0\right) |z_n|^2+ \mathcal{O}\left(|z|^3\right)  \\
&= \text{Re}w_n + \sum_{j, k=1}^{n-1} A_{j, k}  w_j \wb_k + B |z_n|^2 + \mathcal{O}\left(|z|^3\right)  \\
&\leq \text{Re}w_n + \sum_{j, k=1}^{n-1} \tilde A_{j, k}  w_j \wb_k + \mathcal{O}\left(|w_n|^2\right) + \mathcal{O}\left(|w'|^3\right),
\end{split}
\end{equation}
where $w'=\left(w_1, \cdots, w_{n-1}\right)$;
 the first inequality follows from Cauchy-Schwarz inequality and the second inequality follows from the similar argument.
 
For small $\delta\le\delta_p$, where $\delta_p>0$ depends on $r$, the center at $p- \delta N_p$ approximately has the coordinate $w_0=\left(0, \cdots, 0, -\delta\right)$ by ignoring the order of at least 2. Consider $w=\left(u_1, u_2, \cdots, u_{n-1}, -\delta + u_n\right)$ with $|\left(u_1, \cdots, u_{n-1}\right)| < \sqrt{c \delta}$ and $|u_n| < c\delta$. It follows from (\ref{taylor}) that 
\begin{equation}
\label{Rewn}
\begin{split}
r\left(w\right)&\le \text{Re} w_n +  \sum_{j, k=1}^{n-1} \tilde A_{j, k}  w_j \wb_k + \mathcal{O}\left(|w_n|^2\right) + \mathcal{O}\left(|w'|^3\right) \\
&\leq \left(c-1\right) \delta + C |\left(u_1, \cdots, u_{n-1}\right)|^2 +  \mathcal{O}\left(\delta^2\right)  \\
&\leq -\frac{\delta}{2}< 0,
\end{split}
\end{equation}
 if choosing $c<1/100$.  This shows the bidisc $D\left(p-\delta N_p ; c \delta, \sqrt{c\delta}\right)$ is contained in $\Omega\cap U$.
 
By the smoothness of $r$, the conclusion above holds for any boundary points in $\partial\Omega\cap U$ (shrink $U$ if necessary) with the same $\delta_p$ and $c$. By the compactness of $\partial\Omega$, there exists a sufficiently small $\delta_0>0$ and a constant $c>0$, so that the conclusion holds for all $p\in\partial\Omega$. This completes the proof. 
\end{proof}

Now we are ready to present the classical Stein's theorem. Hereafter, let $\pi$ be the normal projection from a small neighborhood $U$ of $\partial\Omega$ to $\partial\Omega$, such that $\text{dist}\left(p, \pi\left(p\right)\right)=\text{dist}\left(p, \partial \Omega\right)$ for any $p \in U$. The set of holomorphic functions on $\Omega$ is denoted by $H(\Omega)$ and the set of Lipschitz continuous functions of order $\alpha$ is denoted by $C^\alpha\left(\Omega\right)$.

\begin{definition}
A $C^2$-smooth curve $\gamma: [0, 1] \rightarrow \Omega$ is a normalized complex tangential curve in $\Omega$ if  $\gamma'\left(t\right) \in \left(T^{\left(1,0\right)}\oplus T^{\left(0,1\right)}\right)_{\pi\left(\gamma\left(t\right)\right)}\partial\Omega$ and $|\gamma'\left(t\right)|\equiv 1$.
\end{definition}

\begin{theorem}[Stein]\label{stein}
Let $n \geq 2$ and $\Omega\subset \mathbb{C}^n$ be a bounded domain with $C^2$ boundary. If $f\in \text{H}\left(\Omega\right) \cap C^\alpha\left(\Omega\right)$ for $0 < \alpha < 1/2$, then 
$f \circ \gamma \in C^{2\alpha}\left( [0, 1]\right)$ for any normalized complex tangential curve $\gamma$. Namely, for some $C>0$ independent of $\gamma$ and $s,t\in[0,1]$,
$$|f\circ\gamma\left(s\right) - f\circ\gamma\left(t\right)| \leq C |s-t|^{2\alpha}.$$
\end{theorem}

\begin{proof}
Fix $U=\{p\in \Omega: \text{dist}\left(p, \partial\Omega\right)<\delta_0\}$ and let $10\epsilon=\delta_0$, where $\delta_0$ is the constant in Proposition \ref{discprop}. For $z', z'' \in E $, where $E:=\{p\in \Omega: \text{dist}\left(p, \partial\Omega\right)>\epsilon\}$ is relatively compact in $\Omega$, $$\left|f\left(z'\right) - f\left(z''\right) \right| \leq \left( \sup_{E} |\nabla f| \right) \cdot |z'-z''|  \leq \left( \sup_{E} |\nabla f| \right) \cdot |z'-z''|^{2\alpha} .$$
By triangle inequality, it then suffices to show that there exists $C>0$, such that  
\begin{equation}\label{holder}
\left| f\left(\gamma\left(h\right)\right) - f\left(\gamma\left(0\right)\right)  \right| \leq C |h|^{2\alpha} 
\end{equation}
 for any the normalized tangential curve $\gamma$ in $U'=\{p\in \Omega: \text{dist}\left(p, \partial\Omega\right)<2\epsilon\}$ and $0 < h < \epsilon$.  
  Let $z'=\gamma\left(0\right)-h^2 N_{\pi\left(\gamma\left(0\right)\right)}$ and $z''=\gamma\left(h\right)-h^2 N_{\pi\left(\gamma\left(h\right)\right)}$. Write
 \begin{equation}
 \begin{split}
 \left| f\left(\gamma\left(h\right)\right) - f\left(\gamma\left(0\right)\right) \right| &\leq   \left| f\left(\gamma\left(h\right)\right) - f\left(z''\right)  \right|  +  \left| f\left(z''\right) - f\left(z'\right)  \right|  +   \left| f\left(z'\right) - f\left(\gamma\left(0\right)\right) \right|       \\
 &= \text{(I) + (II) + (III)} .
 \end{split}
 \end{equation}

 Consider $Q \in U$ with $\text{dist}\left(Q, \partial\Omega\right) :=\delta < 5\epsilon$. Denote the unit outer normal vector and a unit complex tangent vector at $\pi\left(Q\right)$ by $N_{\pi\left(Q\right)}, L_{\pi\left(Q\right)}\in \left(T^{\left(1,0\right)}\oplus T^{\left(0,1\right)}\right)_{\pi\left(\gamma\left(t\right)\right)}\partial\Omega$.
 respectively.
 By Proposition \ref{discprop}, the bidisc 
 in the two dimensional complex linear subspace $\left(N_{\pi\left(Q\right)} \oplus J N_{\pi\left(Q\right)}\right) \oplus \left(\text{Re}L_{\pi\left(Q\right)} \oplus \text{Im}L_{\pi\left(Q\right)}\right)$  centered at $Q$ of bi-radius $\left(c\delta, \sqrt{c\delta}\right)$
  is in $\Omega$ for some $0<c<1/100$. Let $g_{Q}\left(\xi_1, \xi_2\right)=f\left(Q-\xi_1 N_{\pi \left(Q\right)} +\xi_2 L_{\pi \left(Q\right)} \right)$ be a holomorphic function defined
 on the bidisc with $|\xi_1| < c\delta, |\xi_2| < \sqrt{c\delta}$ by identifying $N_{\pi \left(Q\right)}, L_{\pi \left(Q\right)}$ as vectors in $\mathbb{C}^n$. 
  By Cauchy integral formula,
 \begin{equation}\label{de1}
 \begin{split}
 \left| \frac{d}{d t}|_{t=w_1} f \left( Q - t  N_{\pi\left(Q\right)}+w_2 L_{\pi\left(Q\right)} \right)   \right| &= \left| \frac{\partial g_{Q}}{\partial \xi_1}\left(w_1, w_2\right)  \right| \\
 &= \left| \frac{1}{2\pi i}\int_{\partial D_1\left(0, c\delta\right)} \frac{g_{Q}\left(\xi, w_2\right)}{\left(\xi - w_1\right)^2} d\xi       \right|   \\
 &=  \left| \frac{1}{2\pi i}\int_{\partial D_1\left(0, c\delta\right)} \frac{g_{Q}\left(\xi, w_2\right)-g_{Q}\left(w_1,  w_2\right)}{\left(\xi - w_1\right)^2} d\xi       \right|  \\
 &\leq C  \frac{1}{2\pi }\int_{\partial D_1\left(0, c\delta\right)} \left| \xi - w_1 \right|^{\alpha-2} d\xi      \\
 &  \leq C \delta^{\alpha-1}  = C \left(\text{dist}\left(Q, \partial\Omega\right) \right)^{\alpha-1} \\
 & \leq C |w_1|^{\alpha-1}
 \end{split}
\end{equation}
when $|w_1| \leq \frac{c\delta}{2}, |w_2| < \sqrt{c\delta}$.
In particular, by replacing $Q$ by $Q-sN_{\pi\left(Q\right)}$ as the center of the bidisc and by letting $w_1=0$, for $s$ small with $\text{dist}\left(Q-sN_{\pi\left(Q\right)} , \partial\Omega\right)< 4\epsilon$, 
\begin{equation}\label{de2}
\begin{split}
\left| \frac{d}{d t}|_{t=s} f\left( Q - t  N_{\pi\left(Q\right)} \right)  \right|
&= \left| \frac{\partial g_{Q-sN_{\pi\left(Q\right)}}}{\partial \xi_1}\left(0, 0\right) \right| \\
&\leq C\left(\text{dist}\left(Q-sN_{\pi\left(Q\right)} , \partial\Omega\right) \right)^{\alpha-1}\\
& \approx C \left(\delta+s\right)^{\alpha-1} . 
\end{split}
\end{equation}
We note that the same argument also yields
\begin{equation}\label{de6}
\begin{split}
\left| \left( JN_{\pi\left(Q\right)}f  \right) \left( Q-h^2 N_{\pi\left(Q\right)}  \right) \right| &=\left| \frac{d}{d t}|_{t=0} f\left( Q-h^2 N_{\pi\left(Q\right)} - t  JN_{\pi\left(Q\right)} \right)  \right| \\
&\leq C\left(\text{dist}\left(Q-h^2 N_{\pi\left(Q\right)} , \partial\Omega\right) \right)^{\alpha-1} \\
&\approx C \left(\delta+h^2\right)^{\alpha-1} \leq h^{2\alpha-2} ,
\end{split}
\end{equation}
where $J$ is the standard complex structure on $\mathbb{C}^n$.
If   $h^2\le \frac{c\delta}{2}$, as a consequence of (\ref{de1}), 
\begin{equation}\label{wen}
 \left|  \int_0^{h^2}   \frac{d}{d t} f\left( Q - t  N_{\pi\left(Q\right)} \right) dt  \right|   \leq   \int_0^{h^2} \left| \frac{\partial g_{Q}}{\partial \xi_1}\left(-t, 0\right)  \right|  dt  \leq   \int_0^{h^2} C t^{\alpha-1}dt \leq C h^{2\alpha}.
 \end{equation}
If $\frac{c\delta}{2} \leq h^2 \le \epsilon^2$, as a consequence of (\ref{de2}), 
\begin{equation}\label{tian}
 \left|  \int_0^{h^2}   \frac{d}{d t} f\left( Q - t  N_{\pi\left(Q\right)} \right) dt  \right|   \leq  C\left(\delta + h^2 \right)^{\alpha}    \leq C h^{2\alpha}. \end{equation}

On the bidisc with $|\xi_1|< c\delta, |\xi_2| < \sqrt{c\delta}$, Cauchy integral formula yields
\begin{equation}\notag\label{de3}
\begin{split}
&~~~~~~~ \Bigg|\left( L_{\pi\left(Q\right)} N_{\pi\left(Q\right)} f \right)\left(Q-w_1 N_{\pi\left(Q\right)} + w_2 L_{\pi\left(Q\right)}\right)  \Bigg| \\
&= \left| \frac{\partial}{\partial \xi_2}\frac{\partial}{\partial \xi_1} g_Q\left(w_1, w_2\right) \right|  \\
&= \left|  \frac{1}{\left(2\pi i\right)^2}  \int_{\partial D_2\left(0, \sqrt{c\delta}\right)} \left( \int_{\partial D_1\left(0, c\delta\right)} \frac{g_Q\left(\xi_1, \xi_2 \right)- g\left(w_1, \xi_2\right)  }{\left(\xi_1 - w_1\right)^2}  d\xi_1 \right)   \frac{d\xi_2}{\left(\xi_2 - w_2\right)^2}     \right|  \\
&\leq C \delta^{\alpha-1} \delta^{\frac{1}{2}} \left( \delta^{\frac{1}{2}}\right)^{-2} = C  \left(\text{dist}\left(Q, \partial\Omega\right)\right)^{\alpha-\frac{3}{2}},
\end{split}
\end{equation}
when $|w_1| \leq \frac{c\delta}{2},  |w_2|\leq \sqrt{\frac{c\delta}{2}}$. 
In particular, by replacing $Q$ by $Q-sN_{\pi\left(Q\right)}$ as the center of the bidisc and by letting $w_1=w_2=0$, for $s$ small with $\text{dist}\left(Q-sN_{\pi\left(Q\right)} , \partial\Omega\right)< 4\epsilon$, 
\begin{equation}\label{de4}
\Bigg|\left( L_{\pi\left(Q\right)} N_{\pi\left(Q\right)} f \right)\left(Q- s N_{\pi\left(Q\right)} \right)  \Bigg| \leq  C  \left(\text{dist}\left(Q- s N_{\pi\left(Q\right)}, \partial\Omega\right)\right)^{\alpha-\frac{3}{2}} \approx \left(\delta+s\right)^{\alpha-\frac{3}{2}}
\end{equation}

As a consequence, for $h^2\le \epsilon^2$, $Q \in U'$ with $\text{dist}\left(Q, \partial\Omega\right) < 2\epsilon$,
\begin{equation}\label{de4}
\begin{split}
\Bigg|\left( L_{\pi\left(Q\right)} f \right)\left(Q-h^2 N_{\pi\left(Q\right)}\right)  \Bigg| &= \Bigg|  \int_{-\epsilon}^0 \frac{\partial}{\partial t}  \left( \left( L_{\pi\left(Q\right)}\right) f\left(Q -h^2 N_{\pi\left(Q\right)}+ t N_{\pi\left(Q\right)}\right)\right) dt  \\
& \qquad\qquad\qquad\qquad + L_{\pi\left(Q\right)} f\left( Q -h^2 N_{\pi\left(Q\right)}  - \epsilon N_{\pi\left(Q\right)}\right)                      \Bigg|   \\
&\le  \int_{-\epsilon}^0   \left| \left( L_{\pi \left(Q\right)} N_{\pi \left(Q\right)} f \right) \left(Q -h^2 N_{\pi\left(Q\right)}+ t N_{\pi\left(Q\right)}\right)  \right|  dt + M          \\
&\leq C \int_{-\epsilon}^0 \left(\delta  +h^2 - t \right)^{\alpha-\frac{3}{2}} dt + M  \\
&\leq C\left(\delta +h^2\right)^{\alpha-\frac{1}{2}} \leq C h^{2\alpha-1},
\end{split}
\end{equation}
where the second inequality follows from (\ref{de3}) and the third inequality follows as $M=\sup_E |\nabla f|$ is finite. 

Let $\gamma'\left(s\right) = L_{\pi\left(\gamma\left(s\right)\right)}$ and $Q=\gamma\left(0\right) ~\text{or}~\gamma\left(h\right)$. It follows from (\ref{wen}) or (\ref{tian}) that for $ h \leq \epsilon$, $$\text{(I) ~or ~(III)} =\left| \int_0^{h^2}  \frac{d}{d t} f\left( Q - t  N_{\pi\left(Q\right)}\right) dt \right| \leq C h^{2\alpha}.$$
Let $\tilde \gamma\left(s\right) = \gamma\left(s\right) - h^2 N_{\pi\left(\gamma\left(s\right)\right)}$. 
Then $\tilde\gamma'\left(s\right) = \gamma'\left(s\right) - h^2 \frac{d}{ds}{N_{\pi\left(\gamma\left(s\right)\right)}}=a \tilde L_{\pi\left(\gamma\left(s\right)\right)} + b h^2 JN_{\pi\left(\gamma\left(s\right)\right)}$ for $|a|<2, |b|<1$, where $\tilde L_{\pi\left(\gamma\left(s\right)\right)}  \in \left(T^{\left(1,0\right)}\oplus T^{\left(0,1\right)}\right)_{\pi\left(\gamma\left(t\right)\right)}\partial\Omega$.
Applying (\ref{de4}) and (\ref{de6}) to $Q=\gamma\left(s\right)$ for $0 \leq s \leq h$, we have
\begin{equation}\label{wen1}
\begin{split}
\left| \frac{d}{ds} f\circ \tilde\gamma\left(s\right) \right| &\leq 2 \left|\left(\tilde L_{\pi\left(\gamma\left(s\right)\right)} f \right)\left(\gamma\left(s\right)-h^2 N_{\pi\left(\gamma\left(s\right)\right)}\right)  \right| + h^2 \left| \left( JN_{\pi\left(\gamma\left(s\right)\right)}f \right) \left( \gamma\left(s\right)-h^2 N_{\pi\left(\gamma\left(s\right)\right)}  \right) \right| \\
&\leq C \left( h^{2\alpha -1}+ h^{2\alpha}\right)\\
& \leq Ch^{2\alpha -1}.
\end{split}
\end{equation}

Moreover, it follows from (\ref{wen1})  that for $ h \leq \epsilon$, $$\text{(II)} = \left| f\left(\tilde\gamma\left(h\right)\right) -f\left(\tilde\gamma\left(0\right)\right)  \right| \leq  h \cdot \sup_s \left| \frac{d}{ds} f\circ \tilde\gamma\left(s\right) \right| \leq h \cdot C h^{2\alpha-1} = C h^{2\alpha} .$$
The proof of (\ref{holder}) is complete by combining estimates on  (I), (II), (III).

\end{proof}

\begin{remark}
Since complex differentiability and Lipschitz contiuity are local properties and the Riemannian manifold is locally Euclidean, the Stein's theorem holds on the relatively compact domains in complex manifolds equipped with a Riemannian metric.
\end{remark}


\section{Disc Properties and the Local Version of Stein's Theorem}
\label{maintheorem}

We first formulate the following definition.

\begin{definition}
\label{discproperty}
Let $n\ge2$ and $\Omega\subset\C^n$ be a domain with $C^1$ boundary. Let $k\in\Z^+$ and $p\in\partial\Omega$. Denote the unit outward normal vector and a unit complex tangential vector at $p$ by $N_p$ and $L_p$ respectively. Then $\Omega$ is said to satisfy the disc property of index $k$ 
along $L_p$ at $p$, if there are $\delta_p>0$, 
 $c_1
>0$, $c_2
>0$ such that
\[
\left\{Q-w_1N_p+w_2L_p:\,\left(w_1,w_2\right)\in W^{1/k}\right\}\subset\Omega,
\]
for every $Q=p-\delta N_p$ with $0 < \delta\leq \delta_p$, where
\[
W^{1/k}:=\left\{\left(w_1,w_2\right)\in\C^2: |w_1| \leq c_1\delta, |w_2| \leq c_2 \delta^{\frac{1}{k}}\right\}.
\]
\end{definition}

\begin{remark}
By Proposition \ref{discprop}, any bounded domain with $C^2$ boundary satisfies the disc property of index $2$  along every complex tangential direction at each boundary point.
\end{remark}

We are now ready to state our main result, which we refer as  a local (or pointwise) version of Stein's theorem. 

\begin{theorem}
\label{main}
Let $n\ge2$ and $\Omega\subset\C^n$ be a domain with $C^1$ boundary. Assume $\Omega$ satisfies the disc property of index $k$ 
along $L_p$ at $p\in\partial\Omega$ for some $k\in\Z^+$ and some unit complex tangential direction $L_p$. If $f\in \text{H}\left(\Omega\right) \cap C^\alpha\left(\Omega\right)$ for $0 < \alpha < 1/k$, then there exists $C=C\left(c_1, c_2\right)>0$, such that 
\[
|f\left(p-\delta N_p\right)-f\left(p-\delta N_p+hL_p\right)|\le C|h|^{k\alpha}
\]
if $\left(0, h\right) \in \frac{1}{10} W^{1/k}$ and $0 < \delta\leq \delta_p$.
\end{theorem}

\begin{proof}
Let $Q=p-\delta N_{p}$ with $\text{dist}\left(Q, \partial\Omega\right) \approx \delta$. Define $g_Q\left(w_1, w_2\right)=f\left(Q-w_1 N_p +w_2  L_p\right)$ on $W^{1/k}$, where 
\[
W^{1/k}=\left\{\left(w_1,w_2\right)\in\C^2: |w_1| \leq c_1\delta, |w_2| \leq c_2 \delta^{\frac{1}{k}}\right\}.
\] 
By the disc property at $p$, $g_Q\left(w_1, w_2\right)$ is a holomorphic function in $W^{1/k}$. The following argument is similar to the proof of Theorem \ref{stein}. It follows from the Cauchy integral formula that 
$$\left| \frac{\partial g_Q}{\partial \xi_1}\left(w_1, w_2\right) \right| \leq C \left|    \int_{|\xi|=c_1\delta} \frac{g\left(\xi, w_2\right)-g\left(w_1, w_2\right)}{\left(\xi-w_1\right)^2} d\xi \right| \leq C \delta^{\alpha-1}$$
when $|w_1| \leq \frac{c_1\delta}{2}$. This yields $\left| \frac{\partial g_Q}{\partial \xi_1}\left(w_1, w_2\right)  \right| \leq C |w_1|^{\alpha-1} $, which implies 

\begin{equation}\label{t-derivative}
\left| \frac{d}{d t}|_{t=w_1} f\left(Q- t N_p +w_2  L_p\right)   \right| \leq C |w_1|^{\alpha-1}
\end{equation}
when 
$|w_1| \leq \frac{c_1\delta}{2},  |w_2| \leq c_2 \delta^{\frac{1}{k}}$. Moreover, applying Cauchy integral formula on the bidisc $W^{1/k}$,
$$\frac{\partial}{\partial \xi_2} \frac{\partial}{\partial \xi_1} g_Q\left(w_1 w_2\right) = \frac{1}{\left(2\pi i\right)^2} \int_{|\xi_2|=c_2\delta^{\frac{1}{k}}}  \left( \int_{|\xi_1|= c_1\delta} \frac{g\left(\xi_1, \xi_2\right)- g\left(w_1, \xi_2\right)}{\left(\xi_1-w_1\right)^2} d\xi_1 \right)  \frac{d\xi_2}{\left(\xi_2-w_2\right)^2} , $$
it follows that $\left| \frac{\partial}{\partial \xi_2} \frac{\partial}{\partial \xi_1} g_Q\left(w_1 w_2\right) \right| \leq C \delta^{\alpha-1-\frac{1}{k}}\le C|w_1|^{\alpha-1-\frac{1}{k}}$ when $|w_1| \leq \frac{c_1\delta}{2}, |w_2| \leq \frac{c_2\delta^{\frac{1}{k}}}{2}$. This yields
\begin{equation}\label{s-derivative}
\left| \left( L_p N_p f \right) \left(\tilde Q - s N_p \right)   \right| \leq C \left( \text{dist}\left(\tilde Q - s N_p, \partial \Omega\right) \right)^{\alpha-1-\frac{1}{k}},
\end{equation}
for any $\tilde Q=Q-w_1 N_p + w_2 L_p$ with $\left(w_1, w_2\right) \in \tilde W^{1/k}$, where $\tilde W^{1/k}$ is a smaller bidisc with rescaled radii (e.g. $\tilde W^{1/k}=\frac{1}{10}W^{1/k}$), and $|s|\le \frac{c_1\delta}{10}$.

Consider $Q'=Q+h L_p$ for any $|h|$ sufficiently small with $\left(0, h\right) \in \tilde W^{1/k}$ (e.g. $|h|\le\frac{c_2}{10}\delta^{\frac{1}{k}}$). Let $z'=Q-c_0|h|^k N_p$ and $z''=Q'-c_0|h|^kN_p$, such that $\left(0, h\right), \left(c_0|h|^k, 0\right) , \left(c_0|h|^k, h\right) \in W_h^{1/k}\subset \tilde W^{1/k}$, where
\[
W_h^{1/k}=\{\left(w_1,w_2\right)\in\C^2: |w_1| \leq c_1|h|^k/c_2^k, |w_2| \leq |h|\}.
\]
By triangle inequality,  
$$\left|  f\left(Q\right) - f\left(Q'\right)  \right| \leq  \left| f\left(Q\right)-f\left(z'\right)  \right|  +  \left| f\left(z'\right) - f\left(z''\right) \right|   +   \left| f\left(Q'\right) - f\left(z''\right) \right|.$$
By (\ref{t-derivative}),
\[
\left| f\left(Q\right)-f\left(z'\right)  \right|  \leq \int_0^{c_0|h|^k}  \left|  \frac{d}{dt} f\left(Q-t N_p\right)   \right| dt  \lesssim |h|^{k \alpha }
\]
and
\[
 \left| f\left(Q'\right) - f\left(z''\right)   \right|\leq \int_0^{c_0|h|^k}  \left|  \frac{d}{dt} f\left(Q-t N_p+hL_p\right)   \right| dt  \lesssim |h|^{k \alpha }.
\]
It follows from \eqref{s-derivative} that 
\begin{equation*}
\begin{split}
\left| f\left(z'\right) - f\left(z''\right) \right|
&\leq |h| \cdot \sup_{\tilde Q\in Q+\tilde W^{1/k}} \left| L_p f \left( \tilde Q-c_0|h|^kN_p\right)  \right|  \\
&\leq |h| \cdot \sup_{\tilde Q\in Q+\tilde W^{1/k}} \Bigg| \int_{-\frac{c_1\delta_p}{100}}^0\frac{d}{d s} L_p f \left(\tilde Q -c_0|h|^kN_p +s N_p \right)\,ds\\
&\qquad\qquad\qquad\qquad\qquad\qquad +L_p f\left(\tilde Q-c_0|h|^kN_p-\frac{c_1\delta_p}{100}N_p\right)\Bigg|\\
&\lesssim |h|\left(\int_{-\frac{c_1\delta_p}{100}}^0\left|-c_0|h|^k+s\right|^{\alpha-1-\frac{1}{k}}\,ds +M\left(f;p,\delta_p\right)\right)\\
&\lesssim |h|\left(|h|^{k\left(\alpha-\frac{1}{k}\right)}-\left(|h|^k+\frac{c_1\delta_p}{100}\right)^{\alpha-\frac{1}{k}}+M\left(f;p,\delta_p\right)\right)\\
&\lesssim |h|\cdot |h|^{k\left(\alpha-\frac{1}{k}\right)}\\
&=|h|^{k\alpha},
\end{split}
\end{equation*}
where $M\left(f;p,\delta_p\right)$ 
is finite and depends only on $f$, $p$, and $\delta_p$.
\end{proof}

While Theorem \ref{main} states a local result, one can extend it to the semi-local version and we also formulate a 
 semi-local version of the disc property.

\begin{definition}
Let $n\ge2$ and $\Omega\subset\C^n$ be a domain with $C^1$ boundary. 
Let $k\in\Z^+$. 
Then $S \subset \partial\Omega$ is said to satisfy the uniform disc property of index $k$, if there exist $\delta_0=\delta_0\left(k, S\right)>0$, $c_1=c_1\left(k, S\right)>0$, and $c_2=c_2\left(k, S\right)>0$, such that for each $p\in S$, any $L_p\in T_p^{\left(1,0\right)}\oplus T_p^{\left(0,1\right)}$, and any $0<\delta\le\delta_0$,
\[
\left\{Q-w_1N_p+w_2L_p:\,\left(w_1,w_2\right)\in W^{1/k}\right\}\subset\Omega,
\]
where $Q:=p-\delta N_p$ and
\[
W^{1/k}:=\left\{\left(w_1,w_2\right)\in\C^2: |w_1| \leq c_1\delta, |w_2| \leq c_2 \delta^{\frac{1}{k}}\right\}.
\]
\end{definition}

\begin{remark}
By Proposition \ref{discprop}, the entire boundary $\partial\Omega$ of any $C^2$-smooth bounded domain $\Omega$ satisfies the uniform disc property of index $2$.
\end{remark}

\begin{theorem}
\label{semilocalmain}
Let $n\ge2$, $\Omega\subset\C^n$ be a domain with $C^1$ boundary, and $U \subset \partial\Omega$ be a relatively compact, open region in $\partial\Omega$. Assume that $U$ satisfies the uniform disc property of index $k$ for some $k\in\Z^+$. If $f\in H\left(\Omega\right)\cap C^{\alpha}\left(\Omega\right)$ for $0<\alpha<1/k$, then for any normalized complex tangential curve $\gamma:[0,1]\to\Omega$ with $\text{dist}\left(\gamma\left(t\right),\partial\Omega\right)\le\delta_0/5$ for $t\in[0,1]$ and $\pi\left(\gamma\left([0,1]\right)\right)\subset U$,
\[
|f\circ\gamma\left(s\right) - f\circ\gamma\left(t\right)| \leq C |s-t|^{k\alpha},
\]
where $C>0$ is independent of $\gamma$ and $s,t\in[0,1]$.
\end{theorem}

\begin{proof}
It follows essentially the same argument as the proof of Theorem \ref{stein} with estimates \eqref{de1} and \eqref{de4} replaced by \eqref{t-derivative} and \eqref{s-derivative} respectively.
\end{proof}


\section{Worm Domains}
\label{wormdomains}
The worm domains in $\C^2$ are defined as follows (see \cite{CS01}). Let $\beta>\frac{\pi}{2}$, $\varphi: \mathbb{R} \rightarrow \mathbb{R}$ be a smooth even convex function with  
\begin{enumerate}
\item $\varphi\left(x\right) \geq 0;$
\item $\varphi^{-1}\left(0\right) = [-\beta+\frac{\pi}{2},\beta-\frac{\pi}{2} ];$
\item there exists $a>0$ such that $\varphi\left(x\right)>1$ if $x<-1$ or $x>a$;
\item $\varphi'\left(x\right) \not=0$ if $\varphi\left(x\right)=1$.
\end{enumerate}

\begin{definition} 
For each $\beta >\frac{\pi}{2}$, the worm domain is defined by $$\Omega_\beta = \left\{ \left(z_1, z_2\right)\in \mathbb{C}^2 : \left| z_1 + e^{i \log\left(|z_2|^2\right)}\right|^2 < 1-\varphi\left( \log \left(|z_2|^2\right) \right)          \right\}.$$
\end{definition}

\begin{remark}
The defining function of $\Omega_\beta$ can be chosen as $$r_{\beta}\left(z\right) = \left| z_1 +e^{i \log\left(|z_2|^2\right)}\right|^2 -1+\varphi\left( \log \left(|z_2|^2\right)\right).$$
\end{remark}

The pseudoconvexity of the worm domain is well known.

\begin{proposition}
\label{annulus}
$\Omega_\beta$ is a smooth bounded pseudoconvex domain and is strictly pseudoconvex away from the annulus $A:= \left\{\left(0, z_2\right)\in \mathbb{C}^2 : e^{-\frac{\beta}{2}+\frac{\pi}{4}} \leq |z_2| \leq   e^{\frac{\beta}{2}-\frac{\pi}{4}}     \right\}$ .
\end{proposition}

For later purpose, instead of studying $r_{\beta}$ we consider the following more general form of the defining function.

\begin{definition}
Let
\begin{equation} 
\label{definingfunction}
r\left(z_1,z_2\right) = \left|  z_1 \cdot a\left(z_1\right) + b \cdot e^{i \theta\left(z_2\right)}  \right|^2 + \eta\left(|z_2|^2\right),
\end{equation}
where $a: \mathbb{C} \rightarrow \mathbb{C}$ is smooth with $a\left(0\right)\not=0$, $b$ is a nonzero real number, $\eta: \mathbb{R} \rightarrow \mathbb{R}$ is a smooth function with $\eta^{-1}\left(-b^2\right) = I$ for some 
interval $I$ and $\theta: \mathbb{C} \rightarrow \mathbb{R}$ is a smooth function.

Define a domain $\Omega = \left\{ z\in \mathbb{C^2}: r\left(z\right)<0 \right\}$ and $E =  \left\{ \left(0, z_2\right): \eta\left(|z_2|^2\right) = -b^2 \right\}$.
Then $\Omega$ is a smooth domain and $E \subset \partial \Omega$ is a subset in the boundary.
\end{definition}

Note that $E$ may be empty if $I \subset \left(-\infty, 0\right)$.  Additional assumption on $a, b, \eta$ will make $\Omega$ a bounded domain.

\begin{remark}
\label{generalrz1z2}
\,
\begin{enumerate}
\item In particular, when $a\left(z_1\right) \equiv 1$, $\theta\left(z_2\right)= \log\left(|z_2|^2\right)$, $b=1$, and $\eta\left(x\right)= \varphi\left(\log  x\right)-1$, the defining function $r=r_{\beta}$. So $\Omega$ is the worm domain $\Omega_{\beta}$ and $E$ is the annulus $A$ in Proposition \ref{annulus}.
\item The defining function $r$ in \eqref{definingfunction} can be generalized further if $b \cdot e^{i\theta\left(z_2\right)}$ is replaced by a more general function $\phi\left(z_2\right)$, so that $$G:= \left\{ z_2 \in \mathbb{C}: |\phi\left(z_2\right)|=b  \right\}$$ contains $\pi_2\left(E\right) = \left\{ z_2 \in \mathbb{C}: \eta\left(|z_2|^2\right) = -b^2 \right\}$ or $G \cap \pi_2\left(E\right)$ is a domain in $\mathbb{C}$.
\end{enumerate}
\end{remark}

We are interested in the points in $E$ as boundary points in $\partial\Omega$. A direct computation shows: $$r_{z_1}\left(z\right) = \left( a\left(z_1\right) + z_1 a_{z_1}\left(z_1\right) \right) \left( \overline{z_1} \overline{a\left(z_1\right)} + b e^{-i\theta\left(z_2\right)} \right) + \left( z_1 a\left(z_1\right) + b e^{i\theta\left(z_2\right)}  \right) \cdot \overline{z_1} \cdot \overline{a}_{z_1}\left(z_1\right).$$
Since $a\left(0\right)\not=0, b\not=0$,
\begin{equation}
\label{rz1not0}
r_{z_1}|_E = a\left(0\right) b e^{-i\theta\left(z_2\right)}\not=0.
\end{equation}

\begin{lemma}\label{derivative}
For any $m, n \in \mathbb{Z}$ with $m, n \geq 0$ and $m+n \geq 1$, for any $p \in E$, $\frac{\partial^{m+n} r}{\partial z_2^m \partial \bar z_2^n}\left(p\right) =0.$
\end{lemma}

\begin{proof}
Rewrite
\begin{align*}
r\left(z\right) 
&= |z_1 a\left(z_1\right)|^2 + 2 \text{Re}\left(z_1 a\left(z_1\right) b e^{-i\theta\left(z_2\right)}\right) + \left(b^2 + \eta\left(|z_2|^2\right)\right)\\
& = \text{\,\,\,\,\,\,\,\,\,(I)\,\,\,\,\,\,\,\,\,+\,\,\, \,\,\,\,\,\qquad (II) \qquad\,\,\, \,\,\,\,\,+ \qquad\,\,\, (III)}.
\end{align*}

\begin{enumerate}
\item Since (I) is independent of $z_2$, $\frac{\partial}{\partial z_2}$ or $\frac{\partial}{\partial \bar z_2}$ applying to (I) is 0 on $E$.

\item For (II), as $\frac{\partial}{\partial z_2}$ or $\frac{\partial}{\partial \bar z_2}$ will land on $e^{-i\theta\left(z_2\right)}$ or $e^{i\theta\left(z_2\right)}$, the resulting terms always have factors $z_1$ or $\bar z_1$. So  $\frac{\partial}{\partial z_2}$ or $\frac{\partial}{\partial \bar z_2}$ applying to (II) is also 0, when it is evaluating at $p\in E$.

\item Although (III) only depends on $z_2$, $\eta$ on $I$ is identically $-b^2$. So $\frac{\partial}{\partial z_2}$ or $\frac{\partial}{\partial \bar z_2}$ applying to $\eta$ is identically 0 on $I$. Thus $\frac{\partial}{\partial z_2}$ or $\frac{\partial}{\partial \bar z_2}$ applying to (III) is 0 on $E$.
\end{enumerate}
By (1), (2), and (3), it completes the proof.
\end{proof}

\begin{proposition}
Any point in $E$ is not a strongly pseudoconvex point in $\partial \Omega$.
\end{proposition}

\begin{proof}
It follows from Lemma \ref{derivative} that $$r_{z_2}|_E =0.$$ Combining this with \eqref{rz1not0}, one obtains $$\left(r_{z_1}, r_{z_2}\right)|_E = \left(\tau, 0\right),\qquad\text{where}\,\,\,\tau\not=0.$$ This implies that at any point $p \in E$, the holomorphic tangent vector has the form $\left(0, *\right)^t$.

On the other hand, the complex Hessian of $r$ at $p\in E$ is given by $$H_r = \begin{pmatrix} r_{z_1 \bar z_1} & r_{z_1 \bar z_2} \\
r_{z_2 \bar z_1} & r_{z_2 \bar z_2}
\end{pmatrix}=\begin{pmatrix} * & * \\
* &0
\end{pmatrix}$$ by Lemma \ref{derivative}. By acting $H_r$ on the holomorphic tangent vectors, one sees that $p$ is not strongly pseudoconvex by definition.
\end{proof}

\begin{proposition}
\label{discpropatE}
For each $k\in\Z^+$, $\Omega$ satisfies the disc property of index $k$ at any $p\in E$.
\end{proposition}

\begin{remark}
Note that since $\Omega$ is in $\C^2$, there is only one complex tangential direction at each boundary point.
\end{remark}

\begin{proof}
For any $p \in E \subset \partial \Omega$, by Proposition \ref{derivative}, the Taylor series of $r\left(z\right)$ at $p = \left(0, p_2\right)$ expands as the following form:
\begin{align*}
r\left(z\right) 
&= \text{Re} \left( 2 r_{z_1}\left(p\right) z_1 + r_{z_1^2}\left(p\right) z_1^2  + 2r_{z_1 z_2}\left(p\right) z_1 \left(z_2 - p_2\right) + 2r_{z_1 \bar z_2}\left(p\right) z_1 \overline{\left(z_2 - p_2\right)} \right)\\
&\qquad + \frac{1}{2} r_{z_1 \bar z_1}\left(p\right) |z_1|^2 + \cdots.
\end{align*}
By the change of coordinates: $$w_1= r_{z_1}\left(p\right) z_1 +\frac{1}{2}\left( r_{z_1^2}\left(p\right) z_1^2  + 2r_{z_1 z_2}\left(p\right) z_1 \left(z_2 - p_2\right) \right)\qquad\text{and}\qquad w_2 = z_2 - p_2,$$ for each $k \in \mathbb{Z}^+$, the $k$-th order expansion reads:
$$r\left(w\right) = 2 \text{Re}\left(w_1\right) + \sum_* r_{z_1 *} w_1 \mathcal{O}\left(|w_1|, |w_2|\right) + \sum_* r_{\bar z_1 *} \bar w_1 \mathcal{O}\left(|w_1|, |w_2|\right) + \mathcal{O}\left(|w_1|^2, |w_2|^{k+1}\right),$$
where the summation is taken over at least the first order derivatives along $\frac{\partial}{\partial z_1}$ or  $\frac{\partial}{\partial \bar z_1}$.

Take $w_1 \in D\left(-\delta;c_1\delta\right)$, where $D\left(-\delta;c_1\delta\right)$ is a disc centered at $-\delta$ of radius $c_1\delta$ for some $0<c_1<1$ and some $\delta=\delta\left(k,p\right)>0$ sufficiently small such that $Q:=p-\delta N_p \in \Omega$. One obtains $\text{Re}\left(w_1\right)\approx-\delta$ and $|w_1|\approx\delta$. Then for any $|w_2|\leq c_2\delta^{\frac{1}{k}}$ 
\begin{equation}
\label{rw1w2<0}
\begin{split}
r\left(w_1, w_2\right) &\leq -2 \delta + \delta \mathcal{O}\left(|w_2|\right) + \delta^2 + \mathcal{O}\left(|w_2|^{k+1}\right) \\
&\leq -2\delta + \delta \cdot \delta^{\frac{1}{k}} + \delta^2 + \delta^{\frac{k+1}{k}}\\
& <0,
\end{split}
\end{equation}
which implies $Q+W^{1/k}\subset\Omega$, where 
\[
W^{1/k}=\left\{\left(w_1,w_2\right)\in\C^2: |w_1| \leq c_1\delta, |w_2| \leq c_2 \delta^{\frac{1}{k}}\right\}.
\]
So for each $k\in\Z^+$, $\Omega$ satisfies the disc property of index $k$ at $p\in E$.
\end{proof}

For $p\in E$, let $N_p$ and $L_p$ be the unit outward normal vector and the unit complex tangent vector at $p$ respectively. As an application of the disc property at points of infinite type, we have the following corollary. 
\begin{corollary}
\label{steinstheoremonworm}
Let $\Omega$, $p$, and $E$ be as in Proposition \ref{discpropatE}. Assume $f\in \text{H}\left(\Omega\right) \cap C^\epsilon\left(\Omega\right)$ for some $\epsilon \in \left(0, 1\right)$. Then
there exists $\delta_p > 0$ and $c>0$, such that
 $$|f\left(Q\right)-f\left(Q+hL_p\right)|\le C|h|^{\beta}$$ for all $0<\beta<1$, where $Q=p-\delta N_p$ with $0< \delta\le \delta_p$ and $0<|h|<c\delta^{\epsilon} $.
\end{corollary}

\begin{proof}
Pick $k \in \mathbb{Z}^+$ large enough so that $\frac{1}{k} \leq \epsilon$. Thus $f \in C^\alpha\left(\Omega\right)$ for all $0< \alpha < \frac{1}{k}$. By Proposition \ref{discpropatE}, $\Omega$ satisfies the disc property of index $k$ at $p$. Then Theorem \ref{main} applies. 
\end{proof}

\begin{remark}
\label{steinforallpointsinworm}
As in Remark \ref{generalrz1z2}, when $r=r_{\beta}$, $\Omega$ is the worm domain $\Omega_{\beta}$ and $E$ is the annulus $A$. Proposition \ref{annulus} and Corollary \ref{steinstheoremonworm} show that for any holomorphic function, as long as it is Lipschitz continuous, locally it automatically gains in order up to $1$ of Lipschitz continuity in the complex tangential direction at each point in the annulus $A$; whereas, away from the annulus $A$, where points are strongly pseudoconvex, the holomorphic Lipschitz continuous function only gains twice Lipschitz continuity in the complex tangential direction.
\end{remark}


\section{More Applications}
\label{afewremarks}

We investigate the disc property on various classes of domains in $\mathbb{C}^n$.

\begin{enumerate}

\item Let $\Omega = \left\{ z\in \mathbb{C^2}: r\left(z\right)<0 \right\}$, where $r$ is defined as in \eqref{definingfunction}. Let $\kappa>\epsilon>0$, $\kappa>\tau>0$ be small numbers and let $ \chi:\R^+\cup\{0\}\to\R^+$ be a smooth function such that $\chi''\left(t\right)\ge0$ everywhere, $\chi''\left(t\right)>0$ for all $t\in\left(1,1+\kappa\right)$, and
\[
\chi\left(t\right)=\left\{\begin{array}{ll} 1, & t\in[0,1]; \\ 1+\exp\left(-\frac{1}{\left(t-1\right)^{\alpha/2}}\right), & t\in\left(1,1+\epsilon\right); \\ t-\tau, & t\ge 1+\kappa. \end{array}\right.
\]
When $a\left(z_1\right) \equiv 1$, $\theta\left(z_2\right) \equiv 0$, $b=-\sqrt{3}$, and $\eta\left(t\right)= \chi\left(t\right)-4$, then $$r\left(z\right) = |z_1 - \sqrt{3}|^2 + \chi\left(|z_2|^2\right) - 4.$$ So $\Omega$ recovers the domain in \cite{FLZ11}[page 497, Example 1] by shifting $\sqrt{3}$ in $z_1$. By the definition of $\chi$, one sees
\[
E=\{\left(0,z_2\right):\,|z_2|\le1\}.
\]
By the rotational symmetry in $z_1$, $E$ generates the boundary piece:
\[
\tilde E=\{\left(z_1,z_2\right):\, |z_1-\sqrt{3}|^2=3,\,|z_2|\le1\}.
\]
By Proposition \ref{discpropatE} and the flatness of $\tilde E$, $\Omega$ indeed satisfies the uniform disc property of index $k$ on any open subregion $U\subset\tilde E$ for each $k\in\Z^+$. Then Theorem \ref{semilocalmain} applies.

\item Let $\Omega = \left\{ z\in \mathbb{C^2}: r\left(z\right)<0 \right\}$, where $r$ is defined as in \eqref{definingfunction}. Let
\[
\chi\left(t\right)=\left\{\begin{array}{ll} 1-\kappa, & t\le 1-\kappa; \\ \beta\exp\left(-\frac{1}{\left(t^2-\left(1-\kappa\right)^2\right)^{\alpha/2}}\right)+1-\kappa, & t>1-\kappa, \end{array}\right.
\]
where $\alpha\in\left(0,1\right)$, $\kappa>0$ is a small constant such that $\chi\left(t\right)$ is convex on $[0,1]$, and $\beta$ is a constant chosen such that $\chi\left(1\right)=1$, i.e.,
\[
\beta=\kappa\exp\left(\frac{1}{\left(2\kappa-\kappa^2\right)^{\alpha/2}}\right).
\]
As in (1) above, when $a\left(z_1\right)\equiv 1$, $\theta\left(z_2\right)\equiv0$, $|z_1+b|^2$ is replaced by $\chi\left(|z_1|\right)$, and $\eta\left(|z_2|^2\right)= \chi\left(|z_2|\right)-2+\kappa$, then
\[
r\left(z\right)=\chi\left(|z_1|\right)+\chi\left(|z_2|\right)-2+\kappa.
\]
This domain $\Omega$ is also introduced in \cite{FLZ11}[page 511, Example 2]. Two flat pieces in the boundary are: $$P_1=\left\{ |z_1|=1, |z_2|\leq 1-\kappa \right\}\,\,\,\,\text{and}\,\,\,\, P_2=\left\{ |z_2|=1, |z_1|\leq 1-\kappa  \right\}.$$ On $P_1$, $r_{z_1}\not=0$ and $\frac{\partial^{m+n}r}{\partial z_2^m \partial \bar z_2^n} =0$ for $m+n\ge1$ by straightforward calculation.
By a similar argument as in Proposition \ref{discpropatE} and the flatness of $P_1$, $\Omega$ indeed satisfies the uniform disc property of index $k$ on any open subregion $U\subset P_1$ for each $k\in\Z^+$. Then Theorem \ref{semilocalmain} applies. By switching the roles of $z_1$ and $z_2$, the same conclusion holds for $P_2$.

\item Let $\Omega = \left\{ z\in \mathbb{C^2}: r\left(z\right)<0 \right\}$, where $r$ is defined as in \eqref{definingfunction}. As in (1) above, when $a\left(z_1\right)\equiv 1$, $\theta\left(z_2\right)\equiv0$, $|z_1+b|^2$ is replaced by $\text{Re}\left(z_1\right)$, and $\eta\left(x\right)=\exp\left(-1/x\right)$, then
\[
\Omega=\left\{z\in\C^2:\,r\left(z\right)=\text{Re}\left(z_1\right)+\exp\left(-\frac{1}{|z_2|^2}\right)<0\right\}.
\]
Locally around $\left(0,0\right)\in\C^2$, $\Omega$ is similar to the domains $\Omega_1$ and $\Omega_2$ in \cite{FLZ11}[page 504, equation (3.2)]. Note that $\Omega$ has the disc property of index $k$ at $\left(0, 0\right) \in \partial\Omega$ for each $k\in\Z^+$. This can be seen by following the Taylor expansion argument as in Proposition \ref{discpropatE}, since
\begin{equation}
\label{vanishat0}
\left(\exp\left({-\frac{1}{x}}\right)\right)^{\left(n\right)}=0
\end{equation}
for all $n$ at $x=0$. Then Theorem \ref{main} applies.

\item The same technique works for the domain $$\Omega=\left\{z\in\C^2: \exp\left(-\frac{1}{|z_1|^{\alpha_1}}\right) +  \exp\left(-\frac{1}{|z_2|^{\alpha_2}}\right)     
< e^{-1}    \right\},$$ where $\alpha_1,\alpha_2>0$ (cf. \cite{HKR14}, page 3). More precisely, consider the circle $E=\left\{ \left(0, z_2\right): |z_2|=1   \right\} $ in the boundary $\partial \Omega$. By \eqref{vanishat0}, the Taylor expansion of the defining function $$r\left(z\right)=\exp\left(-\frac{1}{|z_1|^{\alpha_1}}\right) +  \exp\left(-\frac{1}{|z_2|^{\alpha_2}}\right)  - e^{-1} $$ at $\left(0, e^{i\theta}\right)$ does not contain any orders of derivatives in $z_1$ and $r_{z_2}\not=0$. Therefore, by a similar argument as in Proposition \ref{discpropatE}, $\Omega$ satisfies the disc property of index $k$ at any point in $E$ for each $k\in\Z^+$. Then Theorem \ref{main} applies.

\item\label{C3example1} For a domain in $\mathbb{C}^3$, not all boundary points of infinite type enjoy the disc property of index $k$ along all complex tangential directions for each $k\in\Z^+$. For example, assume the domain $\Omega$ is locally defined by
\[
\Omega\cap U = \left\{ z\in\C^3: r\left(z\right)=\text{Re}\left(z_1\right)+\exp\left(-\frac{1}{|z_2|^2}\right) + |z_3|^6 <0 \right\}.
\]
At the point of infinite type $\left(0, 0, 0\right) \in \partial \Omega$, $r_{z_1}\not=0$ and $r_{z_2}=r_{z_3}=0$. However, for $k>6$ and $b\neq0$
\[
\left\{\left(-\delta, w_2, w_3\right): |aw_2+bw_3| \leq \delta^{\frac{1}{k}} \right\}
\]
is not contained in $\Omega$. This is because the $z_2$ direction is an "infinite type" direction while the $z_3$ direction is a "finite type" direction. Therefore, the domain $\Omega$ can fit a $\delta^{\frac{1}{k}}$ size disc for each $k$ along $z_2$ direction, but can only fit a $\delta^{\frac{1}{6}}$ size disc along $z_3$ direction.

\item\label{C3example2} Let $\Omega$ be a domain in $\C^3$. Around the origin $0\in\partial\Omega$, the domain is locally defined by
\[
\Omega\cap U = \left\{ z\in\C^3: r\left(z\right)=\text{Re}\left(z_1\right)+\left|z_2^2-z_3^3\right|^2 <0 \right\}.
\]
The origin $0\in\partial\Omega$ is of infinite type, but  is also of finite regular type (cf. \cite{DA82,Ra16}). A similar argument as in \eqref{rw1w2<0} shows that the bidisc $Q+W^{1/6}$ is contained in $\Omega$, where $Q=\left(-\delta,0,0\right)\in\Omega$ for some $\delta>0$ and
\[
W^{1/6}=\left\{\left(z_1,0,z_3\right)\in\C^3: |z_1| \leq c_1\delta, |z_3| \leq c_2 \delta^{\frac{1}{6}}\right\}
\]
for some $c_1$ and $c_2$. So $\Omega$ satisfies the disc property of index $6$ at the origin $0$ along $z_3$ direction. Then Theorem \ref{main} applies.

Moreover, given any point $p\in\partial\Omega$ that is arbitrary close to the origin $0$ that lies on the holomorphic curve $\gamma:\zeta\mapsto\left(0,\zeta^3,\zeta^2\right)$, where $0\neq\zeta\in D(0, \epsilon)\subset \C$ for some small $\epsilon>0$, one sees that $p$ is of infinite regular type (cf. \cite{DA82,Ra16}). By a change of coordinates,
\[
w_1=z_1,\qquad w_2=z_2^2-z_3^3,\qquad w_3=z_3,
\]
the defining function locally can be written as $$r\left(w\right)=\text{Re}\left(w_1\right)+|w_2|^2$$ and $p=\left(0,0,\zeta^2\right)$ for $\zeta\neq0$. A similar argument as in \eqref{rw1w2<0} shows that, for each $k\in\Z^+$ the bidisc $Q+W^{1/k}$ is contained in $\Omega$, where $Q=\left(-\delta,0,\zeta^2\right)\in\Omega$ for some $\delta>0$ and
\[
W^{1/k}=\left\{\left(w_1,0,w_3\right)\in\C^3: |w_1| \leq c_1\delta, |w_3| \leq c_2 \delta^{\frac{1}{k}}\right\}
\]
for some $c_1$ and $c_2$. So $\Omega$ satisfies the disc property of index $k$ at $p$ along $w_3$ direction for each $k\in\Z^+$. Then Theorem \ref{main} applies.

Indeed, the argument above can be generalized to any boundary point $p\in\partial\Omega$ that is in the neighborhood of the origin $0$. Let $p=\left(z_1,z_2,z_3\right)$ such that
\[
r\left(p\right)=\text{Re}\left(z_1\right)+\left|z_2^2-z_3^3\right|^2 =0.
\]
\begin{enumerate}
\item If $p$ is on the line $\{\text{Re}\left(z_1\right)=z_2=z_3=0\}$, then $p$ behaves the same as the origin $0$. So $p$ is of infinite type, but of finite regular type $6$. Also $\Omega$ satisfies the disc property of index $6$ at $p$ along $z_3$ direction. Then Theorem \ref{main} applies.
\item If not both $z_2$ and $z_3$ are $0$, then by a change of coordinates,
\[
w_1=z_1,\qquad w_2=z_2^2-z_3^3,\qquad w_3=\left\{\begin{array}{ll} z_3, & \text{if\,\,}z_2\neq0, \\ z_2, & \text{if\,\,}z_3\neq0, \end{array}\right.
\]
the defining function locally can be written as $$r\left(w\right)=\text{Re}\left(w_1\right)+|w_2|^2.$$ So $\Omega$ satisfies the disc property of index $k$ at $p$ along $w_3$ direction for each $k\in\Z^+$. Then Theorem \ref{main} applies.
\end{enumerate}
This provides a solution to the question raised in \cite{Ra16}[Example 2.3].

\end{enumerate}

\subsection*{Acknowledgement}\, The first author would like to thank J. McNeal for the discussion on Stein's phenomenon during his visit at Ohio State. The first author also would like to thank his thesis advisor S. Krantz for the comments and ideas on this project.

\bibliographystyle{plain}

\end{document}